\documentclass[11pt]{article}
\usepackage{amssymb,amsthm, amsmath, amsfonts, graphicx,latexsym}
\usepackage{fullpage}
\usepackage{multirow}
\newtheorem{lem}{Lemma}
\newtheorem{thm}{Theorem}
\newtheorem{definition}{Definition}
\newtheorem{remark}{Remark}
\newtheorem{corollary}{Corollary}
\usepackage{caption, subcaption}
\usepackage[final]{changes}
\usepackage{placeins}
\usepackage{authblk}
\usepackage[square,sort,comma,numbers]{natbib}
\title{Existence, stability, \added{and symmetry} of relative equilibria with a dominant vortex}
\author[*]{Anna M. Barry}
\author[**]{Alanna Hoyer-Leitzel}
\affil[*]{Department of Mathematics, University of British Columbia}
\affil[**]{Department of Mathematics and Statistics, Mount Holyoke College}

\begin{document}
\maketitle

\begin{abstract}
We analyze existence, stability, and symmetry of point vortex relative equilibria with one dominant vortex and $N$ vortices with infinitesimal circulation.  The dimension of the problem can be reduced by taking an infinitesimal circulation limit, resulting in the so-called $(1+N)$-vortex problem. In this work, we first generalize the reduction to allow for circulations of varying signs and weights.  We then prove that symmetric configurations require equality of two circulation parameters in the $(1+3)$-vortex problem and show that there are stable asymmetric relative equilibria. In a number of examples, we use rigorous methods from algebraic geometry to count all relative equilibria.
\end{abstract}

\section{Introduction}

Point vortex models propose that the motion of small-core, well-separated vortices in a two-dimensional fluid can be described by a set of ordinary differential equations that treats each vortex as a single point, rather than the full partial differential equations for the fluid velocity.  Although there are several point vortex models ranging from geophysical to superfluid literature (see for example \cite{torres2011dynamics,middelkamp2011guiding,lim2001point,morikawa1960geostrophic}), the classical version is derived from the Euler equations and is called the \textit{$n$-vortex problem} \cite{kirchhoff1877vorlesungen,helmholtz1867}.  

One special type of solution to the $n$-vortex problem is the \textit{relative equilibrium}, which is a vortex configuration that appears fixed when viewed in a rigidly rotating frame.  From a physical perspective, stable relative equilibria are of particular interest, as they are most likely to be observed in nature, examples can be found in \cite{kossin2004mesovortices, durkin2000experiments, aref2003vortex}.
The algebraic equations for point vortex relative equilibria are strikingly similar to the analogous set from the $n$-body problem in celestial mechanics, and because of this, the two often exhibit similar relative equilibrium configurations. Examples include the $n$- and $(n+1)$-gons, subjects of Adams prize winning essays by Maxwell \cite{maxwell1859stability} and Thomson (Lord Kelvin) \cite{thomson1883treatise}. These configurations are made up of vortices or masses placed at the vertices of a regular polygon with or without a vortex or mass at the center. Motivated by a conjecture due to Moeckel for the gravitational problem \cite{albouy2013someproblems}, Roberts showed that linearly stable relative equilibria in the $n$-vortex problem are minima of the Hamiltonian restricted to a level set of the angular impulse (the analogue of moment of inertia) when all vortices are spinning in the same direction \cite{roberts2013stability}.  However, if vortex circulations are allowed to have different signs, the topological technique is not as straightforward, as the level surfaces become hyperboloids rather than spheres, and the ``circulation metric'' becomes an indefinite inner product.  This work points to an important difference between the two problems: a vortex can have negative circulation, while the analogous quantity in the celestial mechanics problem, the mass, is nonnegative. 

The circulation of a point vortex is a measure of the rotation of the surrounding fluid, and thus a vortex with large circulation will play a critical role in organizing the flow. With this as motivation, we analyze relative equilibria of the point vortex equations with one strong vortex and $N$ weak vortices, the so-called \textit{$(1+N)$-vortex problem}.  $(1+N)$-point ``mass'' problems have been studied in both the celestial mechanics and vortex communities \cite{hall,moeckel1994linear,casasayas1994central,roberts1998linear,barry2012relative}.  Earlier work on the $(1+N)$-vortex problem assumes that all weak vortices have the same circulation \cite{barry2012relative}.  In this article, we consider the more general case of relative equilibria with a dominant vortex where weak vortex circulations are allowed to have different sizes and signs. Using similar techniques, we reduce the problem to finding critical points of a particular function defined on a circle and show that stability of configurations is determined by eigenvalues of a circulation-weighted version of its Hessian matrix.  

\added{We use algebraic geometric methods, together with the existence and stability results for the $(1+N)$-vortex problem to perform an analysis of symmetry in the $(1+3)$-vortex problem.} Existence of asymmetric relative equilibria has been numerically documented in the literature. For example, Aref and Vainchtein found families of asymmetric relative equilibria by growing new configurations from infinitesimal cases \cite{aref1998point}.  In \cite{newton2007construction}, Newton and Chamoun found asymmetric relative equilibria via calculation of Brownian ratchets. 
Analytical studies of vortex relative equilibria have largely focused on symmetric configurations and/or positive circulation parameters, but there are a few exceptions. One example is the work of Hampton, Roberts, and Santoprete \cite{hampton2014relative}.  Using techniques from algebraic geometry similar to those presented here, they proved the existence of asymmetric configurations in the $4$-vortex problem with two pairs of equal circulations. \added{In addition, Corbera, Cors and Llibre \cite{corbera2011central} classified bifurcations of relative equilibria with two equal masses and found asymmetric cases in the $(1+3)$-body problem.} In \cite{barry2012relative}, there was a surprising identification of \textit{stable} relative equilibria which are not radially symmetric.  In the present work, we have the even more surprising result that fully asymmetric (without a line of symmetry) configurations can be stable. \added{Moreover, we present the novel use of analytical root counting methods in order to verify that numerical calculations find all possible families of relative equilibria.}

The rest of the paper is outlined as follows.  In the next section we introduce the classical $n$-vortex problem and relative equilibria.  In Section \ref{sect2} we define the $(1+N)$-vortex problem and prove results on existence of configurations. Stability is also discussed in this section. In Section \ref{sect4} we introduce the necessary background from algebraic geometry that is then used to analyze the $(1+3)$-vortex problem in detail.  \added{This analysis produces stable, fully asymmetric equilibria, and a number of examples are presented.}  We conclude with a discussion in Section \ref{sect6}.

\section{The $n$-vortex problem and relative equilibria\label{sect1}}
We begin by introducing the equations of motion for point vortices and defining relative equilibrium solutions in the $n$-vortex problem. Relative equilibria are periodic solutions where the vortices organize into a fixed shape that rotates rigidly around the center of vorticity (analogous to center of mass in the point mass gravitational problem). 
Let $q_i=(x_i, y_i) \in \mathbb{R}^2$ be the position of the $i$th vortex with circulation $\Gamma_i$. Let  $J= \left[\begin{array}{cc} 0 & 1 \\ -1 & 0 \end{array} \right]$, and let $\nabla_i$ be the two-dimensional partial gradient with respect to $q_i$. The equations of motion for $n$ vortices are a Hamiltonian system

\begin{equation}\Gamma_i \dot{q}_i = J \nabla_i H(q) \label{eq nvortex}
\end{equation}
 where $\displaystyle H(q)= - \sum_{i<j} \Gamma_i \Gamma_j \log |q_i-q_j|$. 
 
\begin{remark} \added{The point vortex equations are obtained from the vorticity equation for a two-dimensional inviscid fluid by taking the vorticity distribution to be a finite collection of Dirac masses, the ``point vortices'' (see \cite{marchioro1993vortices}).  In contrast to the Newtonian gravitational problem, the conjugate variables for the point vortex equations are the planar position variables (rather than position and momentum) and the system has half the dimension of the n-body problem.}\end{remark}

\begin{definition} A relative equilibrium solution of the $n$-vortex problem with center of vorticity at the origin is a periodic solution with period $2\pi/\omega$ where 
\begin{equation}q_i(t)=e^{-\omega Jt} q_i(0),\; i=1,...,n \end{equation} and
$e^{-\omega Jt}=R^{-1}(t)=\left[\begin{array}{cc} \cos(\omega t) & -\sin(\omega t) \\ \sin (\omega t) & \cos(\omega t) \end{array}\right]$. \end{definition}

Since relative equilibrium configurations appear fixed when viewed from a uniformly rotating coordinate system with rate $\omega$, it is useful to rewrite the equations in these coordinates.  Let
 $\zeta_i(t)=R(t)q_i(t)$. Then 

\begin{align*}\dot{\zeta}_i&=\dot{R}q_i+R\dot{q}_i\\
&=\dot{R}R^{-1}\zeta_i + RJ \tfrac{1}{\Gamma_i}\nabla_i H(q)\nonumber\\
&=\omega J \zeta_i + \tfrac{1}{\Gamma_i}J \nabla_i H(Rq)\nonumber\\
&=\omega J \zeta_i + \tfrac{1}{\Gamma_i}J \nabla_i H(\zeta)
\end{align*}
where in the third line we have used the observation that $R$ and $J$ commute, and that $H$ is invariant with respect to rotations. Thus relative equilibria of \eqref{eq nvortex} are fixed points of the system
\begin{equation} \Gamma_i \dot{\zeta}_i=\omega \Gamma_i J \zeta_i+J \nabla_i H(\zeta_i) \label{rotatingcoords}\end{equation}

\noindent or, more explicitly, they are solutions of
\begin{equation}\omega \Gamma_i \zeta_i^{\perp} + \sum_{j\neq i} \frac{\Gamma_i \Gamma_j (\zeta_j-\zeta_i)^{\perp}}{|\zeta_i-\zeta_j|^2}=0, \label{cc equation} \end{equation}
where $(x,y)^\bot=(-y,x)$.

\section{The $(1+N)$-Vortex Problem \label{sect2}}

In this section, we simplify the question of existence and stability of relative equilibria by exploiting the organizing effect of a single dominant vortex on vortices with infinitesimal strength. We reduce the problem to characterizing critical points of a particular real-valued function $V$ defined on a circle. This function is identified in Theorem 1 and its relationship to existence of relative equilibria is established in Theorem 2. Theorem 3 relates eigenvalues of a weighted Hessian matrix of $V$ to linear stability. This section is a generalization of \cite{barry2012relative} that allows us to introduce the problem of interest and to demonstrate how circulations with varying sizes and signs affect the analysis.  We also introduce a more convenient coordinate system.  \added{The obtained results form the starting point for an analysis of symmetry in the $(1+3)$-vortex problem in Section \ref{sect4}.}

To begin, we formalize what is meant by the phrase ``relative equilibrium of the $(1+N)$-vortex problem.''

\begin{definition} \label{RE def}Let \added{$\varepsilon_k$} be a sequence of real numbers such that \added{$\varepsilon_k \to 0$}
as $k\to \infty$ and let $q_0^k, ..., q_N^k$ be a sequence of relative equilibrium configurations of the $(N+1)$-vortex problem with circulations given by $\Gamma_0^k=1, \Gamma_i^k=\added{\varepsilon_k} \mu_i$, $i=1,...,N$. A \textit{relative equilibrium of the $(1+N)$-vortex problem} is a configuration $\bar{q}_0,...,\bar{q}_N$ such that there exists a sequence of relative equilibria of the $(N+1)$-vortex problem with $q_j^k \to \bar{q}_j$ as $k\to \infty$ (i.e. $\varepsilon_k\rightarrow 0$). \end{definition}

\begin{remark} \added{  The phrase ``$(1+N)$-vortex problem'' refers to a limiting case of the $n$-vortex problem with $n=N+1$, where one strong vortex interacts with $N$ vortices of infinitesimal circulation, and the small circulation parameter tends to zero. The notation is meant to distinguish it from the $(N+1)$-vortex problem; in the latter there is not necessarily a wide separation in relative vortex strengths.  While slightly confusing, the naming scheme is consistent with previous papers from vortex and celestial mechanics literature 
\cite{hall,barry2012relative,casasayas1994central}. }
\end{remark}

\begin{remark}
\added{Throughout the article, we denote a sequence of relative equilibria of the $(N+1)$-vortex problem by $q^k=(q_1^k,...,q_n^k)$, and the limiting configuration by $\bar{q}=(\bar{q}_1,...,\bar{q}_N)$. This notation will carry over to other coordinate systems.}\deleted{ Other terms depending on $\varepsilon_k$ will be denoted with a $k$ subscript. We include the dependence on $k$ through the paper to emphasize the dependence of sequences on $\varepsilon_k$.  }
\end{remark}

\added{  By defining relative equilibria of the $(1+N)$-vortex problem in terms of the infinitesimal circulation limit of a sequence of solutions to \eqref{cc equation}, we retain information about weak vortex-vortex interaction even in the limit $\varepsilon_k \rightarrow 0$. We will see that this enables us to pick out exactly those configurations which can be continued to nonzero circulation.  This approach is in contrast to the restricted problem where $N$ circulations are set to zero and the corresponding $N$ vortices become passive particles under the influence of a single strong vortex. In analogy with Hall's observation for the $(1+N)$-body problem \cite{hall}, the latter setting produces $N$ decoupled two-vortex problems for the interaction of a small vortex with the strong vortex. In general, relative equilibria of the restricted problem will not give rise to relative equilibria of \eqref{cc equation} with nonzero circulation. }

It is possible that as $\varepsilon_k \to 0$, two or more vortices converge to the same limiting position. 
However, we will restrict our study to configurations that do not collide in the limit, and so we require that vortices are bounded away from each other by some $m>0$, i.e. $|q_i-q_j|>m$ for $i\neq j$.  Note that while $m$ may depend on $N$, it is independent of $\varepsilon$.

\begin{remark} \added{In Section \ref{subsect:exist} we will derive a real-valued function of $N$ variables whose nondegenerate critical points are relative equilibria of the $(1+N)$-vortex problem, i.e. sequences of relative equilibria of the $(N+1)$-vortex problem converge to these critical points as $\varepsilon_k\rightarrow 0$, and no two vortices collide in the limit.  We can already identify one such example: if N vortices of equal strength $\Gamma^k=\mu \varepsilon_k$ are placed at the vertices of a regular polygon centered on the remaining vortex with $\Gamma_0=1$, then the system is in relative equilibrium for each $k$. This is the so-called $(N+1)$-gon family.  A straightforward calculation relates the radius $R_k$ of the configuration to $\varepsilon_k$ and the rotation rate $\omega$; one finds $\omega R_k^2=\frac{\mu \varepsilon_k}{2}(N-1)+1$, and so $R_k=\frac{1}{\sqrt{\omega}}+\mathcal{O}(\varepsilon_k)$.  We will see that this radial expansion is not unique to the $(N+1)$-gon configuration, and that convergent sequences of relative equilibria have infinitesimal vortices tending to a circle centered on the strong vortex with 
rate $\mathcal{O}(\varepsilon_k)$ as $\varepsilon_k\rightarrow 0$. }\end{remark}  

\subsection{Heliocentric coordinates }
We now develop a coordinate system that is particularly suited to the problem at hand.
Heliocentric coordinates are often used in the $n$-body problem with one big mass (usually the sun, hence the name ``helio") and several small masses. 
Here, this change of coordinates eliminates the strong vortex from the Hamiltonian equation, thus reducing the number of dimensions by two. This is equivalent to reducing by the integral for center of vorticity. 

The change of coordinates is $Z_0=q_0$ and $Z_i=q_i-q_0$ for $i=1,...,N$, with inverse transformation $F: Z\to q$ given by $q_0=Z_0$ and $q_i=Z_i+Z_0$. The pullback of the symplectic form $\Omega = \sum \Gamma_i dy_i \wedge dx_i$ for this coordinate change is 
\begin{equation*}F^*\Omega=\Gamma_T dY_0 \wedge dX_0 +\sum_{i=1}^N \Gamma_i dY_i\wedge dX_i + \sum_{i=1}^N \Gamma_i (dY_0 \wedge dX_i + dY_i \wedge dX_0) \end{equation*}
where $\Gamma_T=\sum_{i=0}^N \Gamma_i$ is the total circulation.
Let $A_1$  be the $(2N+2)\times (2N+2)$ matrix representation of this symplectic form. The equations of motion for the $(N+1)$-vortex problem  are then 
  \begin{equation*}\dot{Z}=A_1^{-1}\nabla H(Z), \quad H(Z)=-\sum_{j=1}^N \Gamma_0 \Gamma_j \log|Z_j| - \sum_{i<j}\Gamma_i \Gamma_j \log|Z_i-Z_j|\end{equation*}
The coordinates $Z_0$ do not appear in the Hamiltonian.  By fixing the center of vorticity at the origin, $Z_0=-\sum_{i=1}^N \Gamma_i Z_i$, we see that the motion of the strong vortex can be recovered from the motions of the weak vortices.  Thus we can ignore the equations for $Z_0$ entirely and study the remaining system of $2N$ equations. Let $A$ be the lower right $2N \times 2N$ block of $A_1^{-1}$. Then
\begin{equation*}\dot{Z}=A\nabla H(Z). \end{equation*}

In coordinates with $\Gamma_0=1$ and $\Gamma_i=\varepsilon \mu_i$ for $i=1,..,N$, the equations of motion become \begin{equation}\label{eqn:helio}\dot{Z}_i= (1+\varepsilon \mu_i)\frac{Z_i^{\perp}}{|Z_i|^2}+\varepsilon\sum_{j\neq i}\mu_j \left(\frac{ Z_j^{\perp}}{|Z_j|^2} +  \frac{(Z_i-Z_j)^{\perp}}{|Z_i-Z_j|^2} \right).\end{equation}

\subsection{\label{subsect:exist}Existence of Relative Equilibria in the $(1+N)$-vortex problem}

In the following lemma, we prove that relative equilibria of the $(1+N)$-vortex problem exist, i.e. there are sequences of relative equilibria of the $(N+1)$-vortex problem that converge in the limit $\varepsilon_k\rightarrow 0$.

\begin{lem} \label{bounded lemma}
Let $Z^k=(Z_0^k,...,Z_N^k)$ \added{denote a sequence of relative equilibria of \eqref{eqn:helio}} with angular frequency $\omega$, $\Gamma_0=1$, $\Gamma_i=\varepsilon_k \mu_i$, $\varepsilon_k\rightarrow 0$ as $k\rightarrow\infty$.  Then $Z^k$ is bounded, hence there is a subsequence converging to a relative equilibrium of the $(1+N)$-vortex problem.
\end{lem}

\begin{proof}
 Let $\xi_j^k=e^{-i \omega \, t}Z_j^k$ so that $\xi^k=(\xi_1^k,\xi_2^k,...,\xi_N^k)$ is a fixed point of \eqref{eqn:helio} when written in rotating coordinates. We have
\begin{equation*}\dot{\xi}_i^k=\omega J \xi^k_i + (1+\varepsilon_k \mu_i)\frac{(\xi_i^k)^{\perp}}{|\xi_i^k|^2}+\varepsilon_k \sum_{j\neq i} \mu_j\left(\frac{(\xi_j^k)^{\perp}}{|\xi_j^k|^2}+ \frac{(\xi_i^k-\xi_j^k)^{\perp}}{|\xi_i^k-\xi_j^k|^2}\right). \end{equation*}
Multiplying the relative equilibrium equations by $J$ and using the observations that $J^2=-I$, $J(\xi_i^k)^\bot=\xi_i^k$ yields
\begin{align*}
0&= -\omega\xi_i^k+(1+\varepsilon_k \mu_i)\frac{\xi_i^k}{|\xi_i^k|^2}+\varepsilon_k \sum_{j\neq i}\mu_j \left(\frac{\xi_j^k}{|\xi_j^k|^2}+ \frac{(\xi_i^k-\xi_j^k)}{|\xi_i^k-\xi_j^k|^2}\right)\\
\Rightarrow \omega\xi_i^k&=(1+\varepsilon_k \mu_i)\frac{\xi_i^k}{|\xi_i^k|^2}+\varepsilon_k \sum_{j\neq i} \mu_j\left(\frac{\xi_j^k}{|\xi_j^k|^2}+ \frac{\xi_i^k-\xi_j^k}{|\xi_i^k-\xi_j^k|^2}\right).
\end{align*}
Then 
\begin{align}\label{eq:xi}
\omega|\xi_i^k|&\leq(1+\varepsilon |\mu_i|)\frac{1}{|\xi_i^k|}+\varepsilon_k \sum_{j\neq i} |\mu_j|\left(\frac{1}{|\xi_j^k|}+ \frac{1}{|\xi_i^k-\xi_j^k|}\right).
\end{align}

Suppose, by way of contradiction, that there is a subsequence $|\xi^k|\rightarrow\infty$. Then $|\xi_i^k|\rightarrow \infty$ for some $i$. In this case the first term on the right hand side of the inequality \eqref{eq:xi} tends to zero. Since $|\xi_i^k|=|Z_i^k|=|q_i^k -q_0^k|>m$ and $|\xi_i^k-\xi_j^k|=|q_i^k-q_0^k-q_j^k+q_0^k|>m$ \added{for all k}, we see that the second term is bounded:

\begin{equation}\label{eq:bound}
\varepsilon_k \sum_{j\neq i} |\mu_j|\left(\frac{1}{|\xi_j^k|}+ \frac{1}{|\xi_i^k-\xi_j^k|}\right)\leq \frac{2 \varepsilon_k\max\{|\mu_j|\}(N-1)}{m}.
\end{equation}

Thus $|\xi_i^k|$ is bounded for all $k$, a contradiction.

\end{proof}
Since relative equilibria rotate rigidly about the center of vorticity (taken to be zero), we expect the strong vortex to be near the rotational center of the configuration. The next lemma states that this vortex limits to the center of vorticity as $\varepsilon_k\to 0$, and that the infinitesimal vortices limit to a circle centered on the strong vortex.  If we assume that relative equilibria rotate with angular frequency 1, then this picks out the unit circle as the limiting location. \deleted{ The lemma also gives rates of convergence to the circle, which will be important for stability. }

\begin{lem} \label{size lemma}
Set $\omega=1$ and let $Z^k=(Z_0^k,...,Z_N^k)$ be a sequence of relative equilibria of the $(N+1)$-vortex problem that converges to a relative equilibrium $\bar{Z}_0,...,\bar{Z}_N$ of the $(1+N)$-vortex problem as $k\rightarrow\infty$.  Then $|Z_0^k|=\mathcal{O}(\varepsilon_k)$, $|\bar{Z}_0|=0$, $|Z_i^k|^2=1+\mathcal{O}(\varepsilon_k)$, and $|\bar{Z}_i|=1$ for $i=1,...,N$.  
\end{lem}

\added{The proof of Lemma \eqref{size lemma} is a straightforward generalization of Lemma 1 in \cite{barry2012relative}, and we do not include it here.}

Let $\theta=(\theta_1,...\theta_N)$ be an $N$-dimensional vector with entries $\theta_i \in [0,2\pi]$. In the next theorem we identify the function, $V(\theta)$, whose critical points characterize existence and stability of relative equilibria in the problem of one strong and $N$ weak vortices.  In particular, Lemma 2 and Theorem \ref{thm:V} imply that limiting configurations corresponding to relative equilibria of the $(1+N)$-vortex problem have ``weak'' (zero strength in the limit) vortices distributed on the unit circle with angular positions given by critical points of $V(\theta)$.  \added{In later sections, $V$ will be used to simplify the symmetry problem. }

\begin{thm} \label{thm:V}If $(\bar{r},\bar{\theta})=(1,...,1,\bar{\theta}_1,...\bar{\theta}_N)$ is a relative equilibrium \added{(in polar coordinates)} of the $(1+N)$-vortex problem, then $\bar{\theta}$ is a critical point of the function 
 \begin{equation}V(\theta)=-\sum_{i<j} \mu_i \mu_j [ \cos(\theta_i-\theta_j)+\tfrac12 \log(2-2\cos(\theta_i-\theta_j))] \end{equation}
\end{thm}
\begin{proof} Let $Z^{k}=(Z_0^k,...,Z_N^k)$ be a sequence of relative equilibria of the $(N+1)$-vortex problem in heliocentric coordinates that converges to a relative equilibrium $\bar{Z}=(0,\bar{Z}_1,...,\bar{Z}_N)$ of the $(1+N)$-vortex problem as $\varepsilon_k\rightarrow 0$. Let $(r_i^k \cos(\theta_i^k),r_i^k \sin(\theta_i^k))$, $i=1,...,N$ be the polar coordinate representation of $Z_i^{k
}$ \added{and $(\bar{r},\bar{\theta})$ the polar coordinate representation of $\bar{Z}$}. Since relative equilibria rotate rigidly around the center of vorticity \added{(at the origin), we must have $Z_i^k\cdot \dot{Z_i^k}=0$.} Further,

\begin{align*} Z_i^k \cdot \dot{Z}_i^k  &=
 (1+\varepsilon_k \mu_i) \frac{Z_i^k \cdot (Z_i^k)^{\perp}}{|Z_i^k|^2} + \varepsilon_k \sum_{j\neq i} \mu_j \left( \frac{Z_i^k \cdot (Z_j^k)^{\perp}}{|Z_j^k|^2} + \frac{Z_i^k \cdot (Z_i^k-Z_j^k)^{\perp}}{|Z_i^k-Z_j^k|^2} \right)\\
 &=\varepsilon_k \sum_{j\neq i} \mu_j \left( \frac{Z_i^k \cdot (Z_j^k)^{\perp}}{|Z_j^k|^2} - \frac{Z_i^k \cdot (Z_j^k)^{\perp}}{|Z_i^k-Z_j^k|^2} \right)\\
 &=\varepsilon_k \sum_{j\neq i} \mu_j r_i^k r_j^k \sin (\theta_i^k -\theta_j^k) \left( \frac{1}{(r_j^k)^2} - \frac{1}{(r_i^k)^2+(r_j^k)^2-2r_i^k r_j^k \cos (\theta_i^k-\theta_j^k)}\right).
 \end{align*}

Dividing both sides by $\varepsilon_k$ and letting $r_i^k$, $r_j^k \to 1$ as $\varepsilon_k \to 0$ gives
\begin{equation*} 0=\sum_{j\neq i} \mu_j \sin(\bar{\theta}_i-\bar{\theta}_j) \left(1-\frac{1}{2-2\cos(\bar{\theta}_i-\bar{\theta}_j) }\right) , \; i=1,...,N \end{equation*}
Define $V(\theta)=-\sum \mu_i \mu_j \left[ \cos(\theta_i-\theta_j) +\tfrac12 \log(2-2\cos(\theta_i-\theta_j))\right].$ \\
Then $\bar{\theta}=(\bar{\theta}_1,...,\bar{\theta}_N)$ is a solution to the system $\mu_i^{-1} \dfrac{\partial V}{\partial \theta_i}=0$, $i=1,...,N$, and hence it also satisfies $\nabla V=0$.
\end{proof}

\added{
The next theorem states that a given critical point of $V$ is the limit of a sequence of relative equilibria of the $(N+1)$-vortex problem if it satisfies nondegeneracy conditions, thus providing a converse to Theorem \ref{thm:V}.  For the point vortex equations, an unavoidable degeneracy of $V$ comes from rotational symmetry, i.e. any rotation of a critical point of $V$ is again a critical point. Therefore, the Hessian matrix $V_{\theta\theta}$ will have at least one zero eigenvalue associated with the eigenvector $v_0=(1,1,...,1)$.  However, this eigenvalue is found to be harmless, and can be sidestepped by partitioning the nullspace of $V_{\theta\theta}$ into $\text{span}\{v_0\}$ and its complement.  Following the precedent set in \cite{barry2012relative} and \cite{moeckel1994linear}, we \textit{define} a critical point of $V$ to be \textit{nondegenerate} provided it has exactly one zero eigenvalue.  This restriction is enough to guarantee that nondegenerate critical points of $V$ can be continued to relative equilibria of the full point vortex equations via the Implicit Function Theorem.  }

\begin{thm}\label{thm:existence}
Suppose $\bar{\theta}=(\bar{\theta}_1,...,\bar{\theta}_N)$ is a nondegenerate critical point of $V$.  Then for $\bar{r}=(1,1,...,1)$, the configuration $(\bar{r},\bar{\theta})$ is a relative equilibrium of the $(1+N)$-vortex problem.
\end{thm}

\added{The proof follows easily from Theorem 1 in \cite{barry2012relative}.}

\subsection{Linear Stability \label{sect3}}

\added{The function $V$ can also be used to characterize stability of relative equilibria. In Section \ref{sect4}, this will be exploited to show existence of linearly stable asymmetric relative equilibria. The key ingredients for stability are eigenvalues of the ``weighted'' Hessian matrix $\mu^{-1}V_{\theta\theta}$, where $\mu=\text{diag}\{\mu_1,...,\mu_N\}$ is the diagonal matrix containing circulation weights.}

For this section, it is convenient to have rotating Heliocentric coordinates written in polar form.  \added{Let $\Gamma_0=1$, $\Gamma_i=\mu_i\varepsilon$. Then}
	\begin{align}
		\dot{r}_i&=\varepsilon \sum_{j\neq i} \mu_j r_j \sin (\theta_i-\theta_j)\left(\frac{1}{r_j^2}-\frac{1}{r_i^2+r_j^2-2r_i r_j \cos(\theta_i-\theta_j)} \right)\\
		\dot{\theta}_i&=-\omega+(1+\varepsilon \mu_i) \frac{1}{r_i^2} +\varepsilon \sum_{j\neq i} \mu_j \frac{r_i^2 r_j \cos(\theta_i-\theta_j)-r_i r_j^2 \cos (2(\theta_i-\theta_j))}{r_i r_j^2(r_i^2+r_j^2-2r_i r_j\cos(\theta_i-\theta_j))}.
	\end{align}
As before, we take $\omega=1$. 

\added{Let $(r^k, \theta^k)=(r_1^k,...,r_N^k,\theta_1^k,...,\theta_N^k)$ be a sequence of relative equilibria of the $(N+1)$-vortex problem which converges to a relative equilibrium $(\bar{r}, \bar{\theta})=(1,...,1,\bar{\theta}_1,...,\bar{\theta}_N)$ of the $(1+N)$-vortex problem as $\varepsilon_k \to 0$ (i.e. $k\to \infty$). By Theorem 1, $\bar{\theta}$ is a nondegenerate critical point of $V(\theta)$, and by Lemma \ref{size lemma}, $r_i^k=1+\mathcal{O}(\varepsilon_k)$ for $\varepsilon_k$ sufficiently small.  Using this, we find that the linearized Hamiltonian system is of the form $(\dot{\delta r}, \dot{\delta \theta})=M_k (\delta r, \delta \theta)$, where $M_k$ is a Hamiltonian matrix made of four $N \times N$ blocks: }

\begin{equation}M_k=\left( \begin{array}{cc}\varepsilon_k A_k + \mathcal{O}(\varepsilon_k^2) & \varepsilon_k \mu^{-1} V_{\theta\theta}(\theta^k) + \mathcal{O}(\varepsilon_k^2) \\
-2I + \mathcal{O}(\varepsilon_k) & \varepsilon_k D_k +\mathcal{O}(\varepsilon_k^2)\end{array} \right). \end{equation}
Here, $A_k$ and $D_k$ are $N\times N$ matrices of the form:

\begin{align} a_{ii}&=\sum_{j\neq i} \mu_j \frac{\sin(\theta_i^k-\theta_j^k)}{2-2\cos(\theta_i^k-\theta_j^k)} &d_{ii}&=-\sum_{j\neq i} \mu_j \sin(\theta_i^k-\theta_j^k)\\
a_{ij}&=-\mu_j \sin(\theta_i^k-\theta_j^k)&
d_{ij}&=-\mu_j \sin(\theta_i^k-\theta_j^k).
\end{align}

\added{The matrix $M$ has two zero eigenvalues corresponding to the two-dimensional invariant subspace spanned by $v_1=(0,..,0,1,..,1)\in \mathbb{C}^{2N}$ and $v_2=(1,...,1,0,...,0)\in \mathbb{C}^{2N}$, associated to rotational and scaling symmetries of the problem. Thus the diagonalization of $M_k$ has a nontrivial Jordan block, and a relative equilibrium is never a conventionally stable fixed point.  Following Moeckel \cite{moeckel1994linear}, we say a relative equilibrium is \textit{linearly stable} if $M_k$ has purely imaginary eigenvalues with no nontrivial Jordan blocks on a subspace that is skew-orthogonal to the subspace associated with symmetries.}

\added{The next theorem relates eigenvalues of $\mu^{-1}V_{\theta\theta}(\bar{\theta})$ to linear stability of relative equilibria.  }

\deleted{Note that $M$ is Hamiltonian on this complementary subspace in the sense that $\Omega(v,Mw)=-\Omega(Mv,w)$ where $\Omega$ is the polar rotating coordinates form of the skew inner product defined in \added{section on Heliocentric Coordinates}.}

\begin{thm} \label{stability theorem} \added{Let $(r^k, \theta^k)$ be a sequence of relative equilibria of the $(N+1)$-vortex problem that converges to a relative equilibrium $(\bar{r}, \bar{\theta})=(1,...,1,\bar{\theta}_1,...,\bar{\theta}_N)$ of the $(1+N)$-vortex problem as $\varepsilon_k \to 0$, and let $\bar{\theta}$ be a nondegenerate critical point of $V$.
For $\varepsilon_k$ sufficiently small, $(r^k,\theta^k)$ is nondegenerate and is linearly stable if and only if $\mu^{-1}V_{\theta\theta}(\bar{\theta})$ has $N-1$ positive eigenvalues.}
\end{thm}

\added{While the matrix $M_k$ and the stability criterion are different, it turns out that the structure of the proof of Theorem 2 in \cite{barry2012relative} can be used here.  Because of this, we give an outline of the argument and how it generalizes to this setting rather than including all details. }

Let $\lambda_k \in \mathbb{C}$ be nonzero and consider the matrix

\begin{equation}M_k-\lambda_k I= \left( \begin{array}{cc}-\lambda_k I+ \varepsilon_k A_k + \mathcal{O}(\varepsilon_k^2) & \varepsilon_k \mu^{-1} V_{\theta\theta}\left(\theta^k\right) + \mathcal{O}(\varepsilon_k^2) \\
-2I + \mathcal{O}(\varepsilon_k) & -\lambda_k I+ \varepsilon_k D_k+\mathcal{O}(\varepsilon_k^2) \end{array} \right)\label{eqn:mat} \end{equation}

 We can calculate the determinant of $M_k-\lambda_k I$ using the following observation (see e.g. Lemma 4 in \cite{barry2012relative}): If $A, B, C, D$ are $N\times N$ matrices making up the block matrix

\begin{equation} M=\left( \begin{array}{cc} A & B\\
C & D \end{array}\right)\end{equation}

and $A$ is invertible, then 
$\det M=\det(A)\det(D-C A^{-1} B)$.  This gives

\begin{align} \label{eq:det}\det(M_k-\lambda_k I)&= \det\left(-\lambda_k I +\mathcal{O}(\varepsilon_k) \right)\det(-\lambda_k I +\mathcal{O}(\varepsilon_k) \nonumber\\
& - (-2I+\mathcal{O}(\varepsilon_k))(-\lambda_k I+\mathcal{O}(\varepsilon_k))^{-1}(\varepsilon_k\mu^{-1} V_{\theta\theta}(\theta^k)+\mathcal{O}(\varepsilon_k^2))).
\end{align}

\added{Even though the matrix $M_k$ is different, the leading order terms in Equation \eqref{eq:det} match those in \cite{barry2012relative} up to the matrix $\mu^{-1}$. Because $\mu^{-1}$ is independent of $\varepsilon_k$, it is straightforward (but tedious) to check that their calculations carry over to this setting. The result is that there are $2N-2$ roots of $\det(M_k-\lambda_k I)=0$ of the form $\lambda_k=\sqrt{\varepsilon_k}\gamma(\varepsilon_k)$ where $\displaystyle \lim_{\varepsilon_k\rightarrow 0}\gamma(\varepsilon_k)=\pm\gamma_0$ and $\gamma_0$ satisfies}
\begin{equation*}
\det(\gamma_0^2+2\mu^{-1}V_{\theta\theta}(\bar{\theta}))=0.
\end{equation*}

\added{These $2N-2$ eigenvalues, together with the two zero eigenvalues associated with rotational and scaling symmetries, form the entire spectrum of $M_k$. }\deleted{$\bar{M}$, the matrix in the limit $\varepsilon_k\to 0$.}   We thus see that eigenvalues of $V_{\theta\theta}(\bar{\theta})$ are closely related to eigenvalues of $M_k$ for $\varepsilon_k$ small.

\added{To make this relationship more explicit, let $\zeta=-\gamma_0^2$ be a nonzero eigenvalue of $\mu^{-1}V_{\theta\theta}(\bar{\theta})$ and suppose that $\zeta$ is negative or has nonzero imaginary part. Then for $\varepsilon_k$ sufficiently small, $\gamma(\varepsilon_k)$ must have nonzero real part, and therefore $\lambda_k=\sqrt{\varepsilon_k}\gamma(\varepsilon_k)$ has nonzero real part. Thus the relative equilibrium is not linearly stable. Further note that eigenvalues of $2\mu^{-1}V_{\theta\theta}(\bar{\theta})$ have the same signs as the eigenvalues of $\mu^{-1}V_{\theta\theta}(\bar{\theta})$, and that  $\mu^{-1}V_{\theta\theta}(\bar{\theta})$ always has one zero eigenvalue.  This proves the forward direction of the theorem.}

\added{The converse relies on the relationship developed above and a useful Lemma due to Moeckel that exploits the symplectic structure to characterize linear stability \cite{moeckel1994linear}.  The calculations in \cite{barry2012relative} carry over once again, exactly because $\mu^{-1}$ does not affect $\varepsilon_k$-dependent estimates.} 

\added{In the case that $\mu_i>0$ for $i=1,...,N$, $\mu^{-1} V_{\theta\theta}(\bar{\theta})$ having $N-1$ nonzero positive eigenvalues is equivalent to the critical point $\bar{\theta}$ being a nondegenerate minimum of $V$. The following corollary is parallel to a result of Moeckel's for the $(1+N)$-body problem \cite{moeckel1994linear}, where the corresponding ``mass matrix" $\mu^{-1}$ is a positive-definite metric on configurations because masses are positive.  Moreover, it is a direct generalization of the stability result in \cite{barry2012relative}, where all parameters $\mu_i$ were set to one.}

\begin{corollary}\label{cor:min}
Let $\mu_i>0$ for all $i=1,...,N$, and let $(r^k,\theta^k)$, $(\bar{r},\bar{\theta})$ be as in Theorem \ref{stability theorem}.  For $\varepsilon_k$ sufficiently small, $(r^k,\theta^k)$ is nondegenerate and is linearly stable if and only if $\bar{\theta}$ is a nondegenerate minimum of $V$.
\end{corollary}

\added{Corollary \ref{cor:min} is generally not true when some vortices have negative circulations; counterexamples can be seen in Figures 3 and 5.  }The characterization of relative equilibria as nondegenerate minima of an appropriate potential-like function was initially conjectured by Moeckel for the Newtonian problem \cite{albouy2013someproblems}.  In addition to the references mentioned above, Roberts proved the more general result that linear stability in the $N$-vortex problem is equivalent to minimizing the Hamiltonian restricted to a level set of the angular impulse when all vortices have positive circulation \cite{roberts2013stability}.  He also gave counterexamples in the case that circulations have varying signs, which once again points to the complexity introduced by allowing vortices to spin in opposite directions.

\section{The $(1+3)$-Vortex Problem \label{sect4}}
\added{We now} turn to the case $N=3$. \added{The reduction described in Section \ref{sect2} makes analysis of existence, stability, and symmetry more manageable through a characterization of critical points of $V$. }\deleted{The gradient of $V$ is a system of equations in three position variables and three circulation parameters. While we could attempt to solve for the critical points numerically, } We are able to obtain robust analytical results by using techniques from computational algebraic geometry, namely Gr\"{o}bner bases and the Hermite method. We discover conditions under which critical points of $V$ are symmetric, and this yields \textit{asymmetric} relative equilibria of the $4$-vortex problem. Theorem \ref{stability theorem} is then used to characterize stability. Moreover, we give a rigorous count of the number of critical points of $V$ (up to rotational symmetry). \deleted{using the Hermite Method., and hence we know that numerical methods are finding all possible families of relative equilibria.}

The next \added{two subsections} give a brief technical overview of two important theories from algebraic geometry that we apply in the rest of the section. \deleted{Mathematica} The first theory introduces the Gr\"{o}bner basis, which transforms a given a set of polynomial equations into a second set of polynomials with the same solution set. Combining a Gr\"{o}bner basis with an elimination ideal projects the set of solutions onto a subspace of the variables. \added{In Section \ref{sect:sym} we make a change of coordinates that transforms the equations for critical points of $V$ into polynomial equations, and then use this method to project the set of solutions from the space of both position variables and circulation parameters on to circulation parameter space, thus reducing the dimension of the problem.} The second technique, the Hermite Method, is a root-counting algorithm for systems of polynomial equations with coefficients in $\mathbb{Q}$,\added{ and we use it to count critical points of $V$ in later subsections.} \added{A deep understanding of Sections 4.1 and 4.2 is not needed to follow the rest of the section; most of the difficulty involves setting up the problem in such a way that the algebraic geometry can be successfully implemented.}


\deleted{Below we outline the Gr\"{o}bner basis and Hermite method techniques, which are used to describe the variety (zero set) of a set of polynomials.}

\subsection{The Gr\"{o}bner Basis}
The Gr\"{o}bner Basis is an incredibly useful tool for solving polynomial equations. A Gr\"{o}bner Basis can also be used to eliminate variables from a set of equations, either naturally through its algorithmic existence or by finding the Gr\"{o}bner basis of the elimination ideal, which gives the projection of the solution onto a \added{subspace}. \deleted{SAID THIS IN PREVIOUS PAR ABOVE/REDUNDANT For our results on relative equilibria, after changing the equations for critical points of $V$ to polynomial equations, a Gr\"obner Basis with an elimination ideal is calculated to eliminate the position variables of the vortices so that we can analyze critical points of $V$ in only the circulation parameter space.}\added{We now introduce the Gr\"obner basis and relevant results.}\deleted{We now briefly outline some results from algebraic geometry related to the set of roots for a system of polynomial equations.} For more details and proofs of these well-known theorems, see \cite{Cox_Ideals}. 

Let $k$ be a field, let $F$ be a set of polynomials in $k[x_1,...,x_n]$, and let $I=(F)$ be the ideal generated by $F$. Geometrically, the set of zeros of $F$ is called the \textbf{variety} of $F$ and denoted $Var(F)$. 

\begin{thm}If $I=(F)$ is the ideal generated by $F$ in $k[x_1,...,x_n]$, then $Var(I)=Var(F)$. \end{thm}

\noindent 

For a given ideal, the generating set may not be unique. In order to simplify the problem, we will look at a specific generating set for $I=(F)$.

A monomial order, as defined next, is an example of lexicographic order on monomials. It can refer to either the ordering on monomials or to the ordering of the vectors of exponents of the monomials. 

\begin{definition} A \textbf{monomial order} is a total ordering on monomials $x_{1}^{\alpha_{1}}\cdots x_{n}^{\alpha_{n}}$, $\alpha_i \in \mathbb{Z}_{\geq0}$ in a polynomial ring such that if $\alpha =(\alpha_1,...,\alpha_n)> \beta=(\beta_1,...\beta_n)$, then $\alpha+\gamma>\beta+\gamma$ for any $\gamma \in \mathbb{Z}_{\geq0}^n$, and the ordering is a well-ordering. Moreover, an \textbf{elimination ordering} is a monomial ordering where $x_1^{\alpha_1}\cdots x_k^{\alpha_k}> x_{k+1}^{\beta_{k+1}}\cdots x_n^{\beta_n}$ whenever one of $\alpha_i>0$ for $i=1,...,k$. \end{definition}

\begin{definition} A \textbf{Gr\"{o}bner basis} for an ideal $I$ is a set of polynomials that generate the ideal $I$ and for all $ f \in I$ there is some $g$ in the basis such that the leading term of $g$ divides the leading term of $f$. \end{definition}

\begin{thm} Every ideal in $k[x_1,...,x_n]$ has a Gr\"{o}bner basis with respect to a given monomial ordering. \end{thm}

Buchberger's Algorithm allows us to find a Gr\"{o}bner basis, often giving a polynomial in the basis with only one variable (or at least fewer variables than the original polynomial). This ``elimination" is based on the choice of monomial order. A Gr\"{o}bner basis can also be used to guarantee elimination of \added{specified} \deleted{certain} variables from a system of polynomial equations, as evidenced by the following theorem. 

\begin{thm}Let $G$ be a Gr\"{o}bner basis for $I \subset k[x_1,..,x_k,x_{k+1},...,x_n]$ with the elimination ordering $x_1^{\alpha_1}\cdots x_k^{\alpha_k} > x_{k+1}^{\beta_{k+1}}\cdots x_n^{\beta_n}$. Then $G\cap I_k$ is a 
Gr\"{o}bner basis for the elimination ideal $I_k = I \cap k[x_{k+1},...,x_n]$.\end{thm}

\subsection{Hermite Method}

In addition to the Gr\"{o}bner basis, we use the Hermite method for counting roots of a set of polynomials $P$.  \added{The method utilizes the equivalent calculation of finding the signature of a particular quadratic form. This quadratic form is constructed using traces of linear maps given by multiplication of elements in the basis of the quotient ring $\mathbb{Q}[x]/I$, where $I$ is the ideal generated by $P$. } Below we give a short outline of the construction of this matrix. Another description of the method for one polynomial can be found in \cite{Coste_SemiAlgebraicGeometry}.  A more detailed description of the method for a set of polynomials and its proof is given in Chapter 2 of \cite{Cox_UsingAlg} or Chapter 4 of \cite{basu_algorithms}. \added{The important result that will be used when counting relative equilibria in later sections is stated in Theorem \ref{hermite theorem}.} 

Let $I$ be the ideal generated by a set of polynomials in $\mathbb{Q}[x_1,...,x_n]$. The quotient ring $\mathbb{Q}[x_1,...,x_n]/ I$ is a vector space over $\mathbb{Q}$ and it is finite dimensional (if $Var(I)$ is finite). The Hermite method of counting real roots involves identifying a particular matrix with entries constructed from basis elements of the vector space $\mathbb{Q}[x_1,...,x_n]/ I$.

To construct this matrix, one first finds the Gr\"{o}bner basis $G$ of the ideal $I$ with respect to a chosen 
lexicographic order. The basis $B$ for $\mathbb{Q}[x_1,...,x_n]/ I$ is a set of monomials that are not in the ideal of leading terms $\langle LT(I)\rangle$ of $G$ (with respect to that same lexicographic order), \[B=\{x^{\alpha} : x^{\alpha}\notin\langle LT(I)\rangle \}.\]

This basis consists of all monomials with exponents less than the exponents of the leading terms of the Gr\"{o}bner basis. For a basis element $f_i \in B$, there is a corresponding linear map $m_i$ for multiplication by $f_i$ over the vector space $\mathbb{Q}[x_1,...,x_n]/ I$.
Define the matrix $\mathcal{H}(I)$ by the entries $\mathcal{H}_{ij}=Tr(m_i \cdot m_j)$.  

The signature of a quadratic form is the difference between the dimensions of the positive definite and negative definite subspaces. This can be calculated for the matrix representation of the form, either by calculating the number of positive and negative eigenvalues of the matrix, or through equivalent calculations.

\begin{thm} \label{hermite theorem}The signature of the matrix $\mathcal{H}(I)$ is the number of distinct real roots of the polynomials generating $I$. The rank of $\mathcal{H}(I)$ is the number of distinct roots over $\mathbb{C}$. \end{thm}

\added{We will illustrate the implementation of the Hermite method in Section \ref{sect:eq}.}
\deleted{In our examples, the Hermite method was implemented in Mathematica. The code can be found in Appendix \ref{Appendix}.}

\subsection{Symmetry \label{sect:sym}}

In the $(1+3)$-vortex problem, symmetry occurs when a line can be drawn through the strong vortex and one infinitesimal vortex such that the other infinitesimal vortices are symmetrically located with respect to this line.  In the following theorem, we show that symmetric configurations of the $(1+3)$-vortex problem must have two infinitesimal vortices with equal weight. \added{Equivalently, any relative equilibrium of the $(1+3)$-vortex problem with no equal weights will be asymmetric.}

\begin{thm} If a relative equilibrium of the $(1+3)$-vortex problem is symmetric, then two of the infinitesimal vortices have equal weight $\mu_i=\mu_j$, $i\neq j$. \label{symmetry theorem}
\end{thm}

\begin{proof} Consider the equations for critical points of $V$. We first set $\theta_1=0$ to eliminate rotational symmetry. The partial derivatives of $V$ are linearly dependent since $V_{\theta_1}=-V_{\theta_2}-V_{\theta_3}$, and so we need only find the roots of $V_{\theta_2}$ and $V_{\theta_3}$. The two equations of interest are

\begin{align}
V_{\theta_2}&=-\mu_1 \mu_2 \left(-\sin(\theta_2) + \frac{\sin(\theta_2)}{2 - 2 \cos(\theta_2)} \right)- 
 \mu_2 \mu_3 \left(-\sin(\theta_2 - \theta_3) + \frac{\sin(\theta_2 - \theta_3)}{2 - 2 \cos(\theta_2 - \theta_3)}\right) \label{partial1}\\
 V_{\theta_3}&= -\mu_2 \mu_3 \left(\sin(\theta_2 - \theta_3) - \frac{\sin(\theta_2 - \theta_3)}{2 - 2 \cos(\theta_2 - \theta_3)}\right) - 
 \mu_1 \mu_3 \left(-\sin(\theta_3) + \frac{\sin(\theta_3)}{2 - 2 \cos(\theta_3)}\right). \label{partial2}
 \end{align}

There are three possible symmetric configurations, and we treat each case separately.\\

\noindent{\it Case 1:} $\theta_3=2\pi-\theta_2$. We make this substitution into equations \eqref{partial1} and \eqref{partial2} and then expand using identities. Setting aside the denominators for now, we see that the numerators are  
\begin{align}
V_{\theta_2}^{\textrm{num}}&=\mu_2 (\mu_1 - 2 \mu_1 \cos(\theta_2) + 2 \mu_3 \cos(\theta_2) - \mu_1 \cos^2(\theta_2) - 
   2 \mu_3 \cos^2(\theta_2) + 2 \mu_1 \cos^3(\theta_2) \nonumber \\& - 4 \mu_3 \cos^3(\theta_2)  +4 \mu_3 \cos^4(\theta_2) +
    \mu_1 \sin^2 (\theta_2)- 2 \mu_1 \cos(\theta_2) \sin^2(\theta_2) + 4 \mu_3 \cos(\theta_2) \sin^2(\theta_2)  \\
   &- 
   4 \mu_3 \cos^2(\theta_2) \sin^2(\theta_2))\nonumber
   \end{align}
and
\begin{align}
V_{\theta_3}^{\textrm{num}}&=-\mu_3 (\mu_1 - 2 \mu_1 \cos(\theta_2) + 2 \mu_2 \cos(\theta_2) - \mu_1 \cos^2(\theta_2) - 
   2 \mu_2 \cos^2(\theta_2) + 2 \mu_1 \cos^3(\theta_2)\nonumber\\ & - 4 \mu_2 \cos^3(\theta_2) + 4 \mu_2 \cos^4(\theta_2) +
    \mu_1 \sin^2(\theta_2) - 2 \mu_1 \cos(\theta_2) \sin^2(\theta_2) + 4 \mu_2 \cos(\theta_2) \sin^2(\theta_2) \\& - 
   4 \mu_2 \cos^2(\theta_2) \sin^2(\theta_2))\nonumber
\end{align}

The denominators of these two equations are the same: $2 (-1 + \cos(\theta_2)) (-1 + \cos(2 \theta_2))$. \added{The values $\theta_2=0,\, \pi$ correspond to collisions which we do not consider here. We therefore ignore the denominators and focus only on the numerators $V_{\theta_2}^{\textrm{num}}$ and $V_{\theta_3}^{\textrm{num}}$.}

Next, we apply a change of coordinates that transforms the system into rational polynomial equations. The transformation is motivated by the tangent half-angle identities for sine and cosine:

 \begin{equation} \cos(\theta_2)= \frac{r^2-1}{1+r^2} \qquad \sin(\theta_2)=\frac{2r}{1+r^2} \label{coord change}. \end{equation}

Conveniently, terms involving cosine and sine are now written in terms of a single variable, and the collisions terms where $\theta_2=0,\,\pi$ have been moved to infinity. Thus the denominators of the new equations can be ignored, resulting in two polynomials:
\begin{align}
V_{\theta_2}^*(r)&=-4 \mu_2 (-\mu_3 - 6 \mu_1 r^2 + 15 \mu_3 r^2 - 4 \mu_1 r^4 - 15 \mu_3 r^4 + 2 \mu_1 r^6 + 
   \mu_3 r^6)\\
V_{\theta_3}^*(r)&=4 \mu_3 (-\mu_2 - 6 \mu_1 r^2 + 15 \mu_2 r^2 - 4 \mu_1 r^4 - 15 \mu_2 r^4 + 2 \mu_1 r^6 + 
   \mu_2 r^6).
   \end{align}

Since we are looking for circulation weights $\mu_i$ that allow for symmetric equilibria, we \added{will} project the variety of the ideal $\langle V_{\theta_3}^*,V_{\theta_3}^*\rangle$ onto $(\mu_1,\mu_2,\mu_3)$-space. \added{By choosing an appropriate ordering, the Gr\"{o}bner Basis theory can be applied to calculate the elimination ideal that eliminates the configuration variable $r$}. \added{Given the two polynomials  $V_{\theta_3}^*$ and $V_{\theta_3}^*$, the theory is implemented via Buchberger's Algorithm.  This can be done in Mathematica with the $\texttt{GroebnerBasis}$ command and} elimination ordering $r>\mu_1>\mu_2>\mu_3$, \added{which gives the basis} $\mu_1\mu_2\mu_3(\mu_2-\mu_3)$. Hence (nontrivial) zeroes occur only when $\mu_2=\mu_3$. \\

\noindent{\it Case 2:} $\theta_3=2\theta_2$. The same process results in the Gr\"{o}bner basis  $\mu_1\mu_2\mu_3(\mu_1-\mu_3)$ for the elimination ideal. Thus symmetric configurations occur when $\mu_1=\mu_3$.\\

\noindent{\it Case 3:} $\theta_2=2\theta_3$.  In this case, the Gr\"{o}bner basis for the elimination ideal is $\mu_1\mu_2\mu_3(\mu_1-\mu_2)$, and symmetric configurations occur when $\mu_1=\mu_2$. 
\end{proof}

\subsection{Equal Weights\label{sect:eq}}  
Let $\mu_1=\mu_2=\mu_3$. The critical points of $V$ depend only on the ratio $\mu_1: \mu_2 : \mu_3$. To apply the Hermite method we must use integer values of $\mu_i$, hence we assume $\mu_1=\mu_2=\mu_3=1$. We will now give an illustration of the implementation of the Hermite Method, which will also be applied in Section \ref{sect:example 2}. As in the proof of Theorem \ref{symmetry theorem}, we set $\theta_1=0$ to reduce by rotational symmetry and replace $\cos(\theta_2)$, $\sin(\theta_2)$, $\cos(\theta_3)$, and $\sin(\theta_3)$ with $r_2$ and $r_3$ according to the tangent half-angle identities (see equation \eqref{coord change}).  Moreover, we again find that the denominators of the resulting rational functions are only zero at collisions, and so we ignore them.  The numerators are
\begin{align} V_{\theta_2}^{num}(r_2,r_3)=&-8 (r_2 - r_3) (-\mu_3 - 3 \mu_1 r_2^2 + 3 \mu_3 r_2^2 + \mu_1 r_2^4 + 3 \mu_1 r_2 r_3 - 
   9 \mu_3 r_2 r_3 - \mu_1 r_2^3 r_3 + 3 \mu_3 r_2^3 r_3 \nonumber \\ &+ 3 \mu_3 r_3^2 - 
   3 \mu_1 r_2^2 r_3^2 - 9 \mu_3 r_2^2 r_3^2 + \mu_1 r_2^4 r_3^2 + 3 \mu_1 r_2 r_3^3 + 
   3 \mu_3 r_2 r_3^3 - \mu_1 r_2^3 r_3^3 - \mu_3 r_2^3 r_3^3)\\
V_{\theta_3}^{num}(r_2,r_3) =& -8 (r_2 - r_3) (\mu_2 - 
   3 \mu_2 r_2^2 - 3 \mu_1 r_2 r_3 + 9 \mu_2 r_2 r_3 - 3 \mu_1 r_2^3 r_3 - 
   3 \mu_2 r_2^3 r_3 + 3 \mu_1 r_3^2 - 3 \mu_2 r_3^2 \nonumber \\&+ 3 \mu_1 r_2^2 r_3^2 + 
   9 \mu_2 r_2^2 r_3^2 + \mu_1 r_2 r_3^3 - 3 \mu_2 r_2 r_3^3 + \mu_1 r_2^3 r_3^3 + 
   \mu_2 r_2^3 r_3^3 - \mu_1 r_3^4 - \mu_1 r_2^2 r_3^4).
\end{align}
Since the common factor $(r_2-r_3)$ also corresponds to collisions, we remove it from both polynomials. With $\mu_i=1$, these become
\begin{align}
p(r_2,r_3)&=1 - r_2^4 + 6 r_2 r_3 - 2 r_2^3 r_3 - 3 r_3^2 + 12 r_2^2 r_3^2 - r_2^4 r_3^2 - 
 6 r_2 r_3^3 + 2 r_2^3 r_3^3 \\
q(r_2,r_3)&= -1 + 3 r_2^2 - 6 r_2 r_3 + 6 r_2^3 r_3 - 
 12 r_2^2 r_3^2 + 2 r_2 r_3^3 - 2 r_2^3 r_3^3 + r_3^4 + r_2^2 r_3^4.
 \end{align}
We then calculate the Gr\"{o}bner basis for the ideal $I$ generated by $p$ and $q$ using the $\texttt{GroebnerBasis}$ command in Mathematica and the monomial ordering $\texttt{DegreeReverseLexicographic}$. The Gr\"obner basis for $I$ contains 7 polynomials with leading terms 
 \begin{equation} LT(I)=\{ 
 r_2^3r_3^3,\,\,r_2^4r_3^2, \,\,r_2^5r_3^1, \,\,r_2^6, \,\,r_3^7, \,\,r_2r_3^6, \,\,r_2^2r_3^5\} .\end{equation} 
The quotient ring $Q[r_2,r_3]/ I$ has a basis of 24 monomials that are not in $\langle LT(I)\rangle$, the ideal generated by the leading terms. This basis consists of all monomials with exponents less than those of the leading terms with respect to Degree Reverse Lexicographic ordering. The basis monomials $\{b_1,...,b_{24}\}$ are ordered with respect to the same monomial ordering. 
We are now in a position to construct the matrix $\mathcal{H}$ whose signature gives the number of distinct real roots of $V_{\theta_2}$ and $V_{\theta_3}$ (see Theorem 7).  To do so, we must first create the linear map of multiplication of any polynomial $f \in Q[r_2,r_3]/ I$ by two basis monomials $b_ib_j$. The trace of this map becomes the $ij$-th entry in $\mathcal{H}$.  This calculation was done using the $\texttt{HermiteForm}$ command as defined in Appendix \ref{Appendix}. 

The Hermite method outlined above yields 14 real-valued critical points of $V$ (up to rotational symmetry), all of which \added{have a line of symmetry}. We explicitly identify these using Equations \eqref{partial1} and \eqref{partial2}. In each case one vortex lies on the line of symmetry, and the other two are separated from it by an angle of $\frac{\pi}{4}$, $\frac{2\pi}{3}$, or $\frac{3\pi}{4}$.  Up to ordering of the vortices, these are the 3 families found in \cite{barry2012relative}. For all 14 critical points, we use Theorem \ref{stability theorem} to characterize stability.  The configurations and associated linear stability are summarized in Figure \ref{example1families}. Notice that the stable configurations correspond to minima of $V$, and unstable configurations correspond to maxima or saddle points. Because of this, we sometimes call $V$ a ``limit potential.'' \added{As pointed out in \cite{barry2012vortex}, relative equilibria resembling those in Figure \ref{family3} below have been observed in electron column experiments (see Figure 7 in \cite{durkin2000experiments}), where the physical equations of motion for the columns are equivalent to the two-dimensional inviscid vorticity equations.  Since these configurations are minima of $V$, they correspond to stable relative equilibria.}
\begin{figure}[ht]
	\centering
\begin{subfigure}[t]{\textwidth}
\centering
		\includegraphics[width=1in]{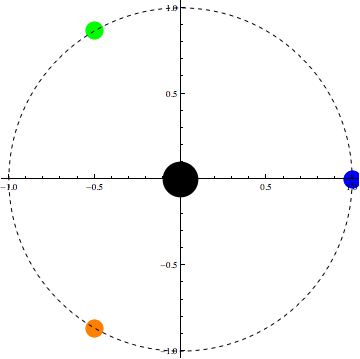} \qquad \qquad \includegraphics[width=1 in]{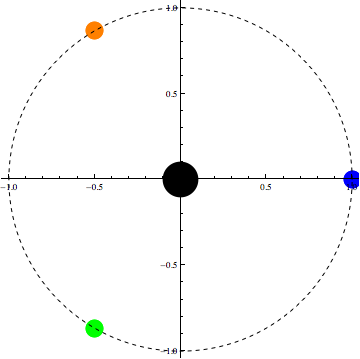}
		\caption{Family 1: \added{Two relative equilibria of the $(1+3)$-vortex problem that are maxima of $V$. They are limits of sequences of unstable relative equilibria, and so we \added{abuse terminology and refer to} them as unstable.}}
		\label{family1}
\end{subfigure}

\begin{subfigure}[t]{1\textwidth}
\centering
\includegraphics[width=1in]{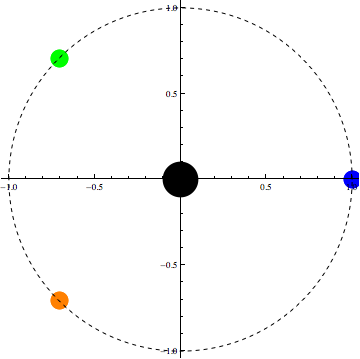}
\includegraphics[width=1in]{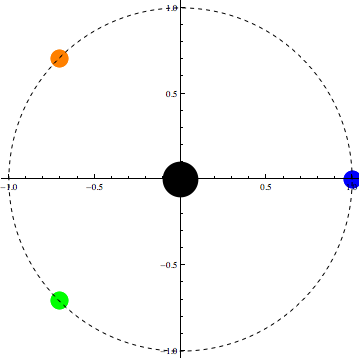}
\includegraphics[width=1in]{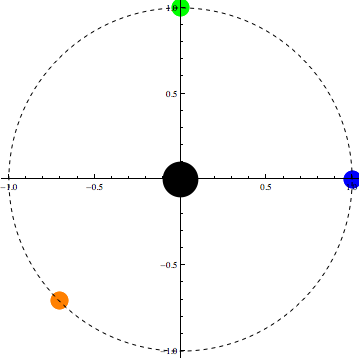}
\includegraphics[width=1in]{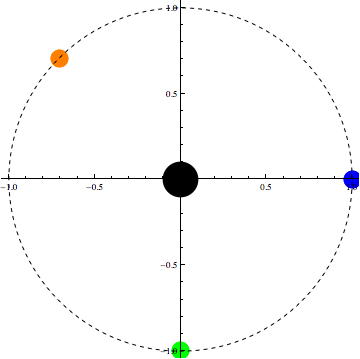}
\includegraphics[width=1in]{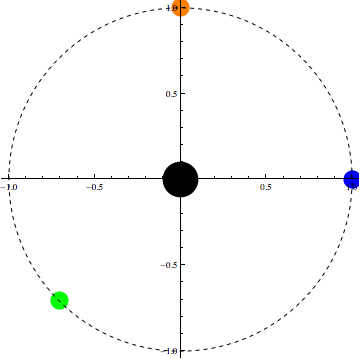}
\includegraphics[width=1in]{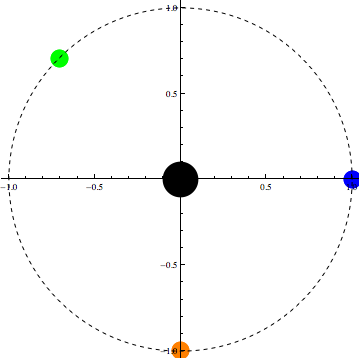}
\caption{Family 2: Six saddle points of $V$.  These configurations are unstable.}
\label{family2}
\end{subfigure}
\begin{subfigure}[t]{\textwidth}
\centering
\includegraphics[width=1in]{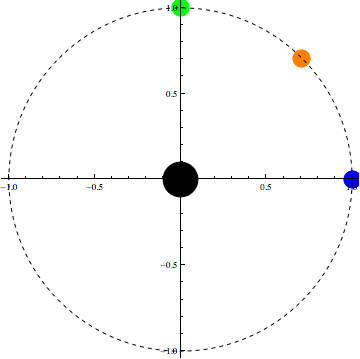}
\includegraphics[width=1in]{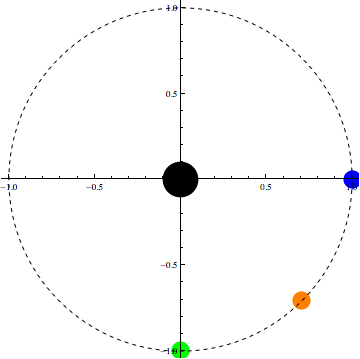}
\includegraphics[width=1in]{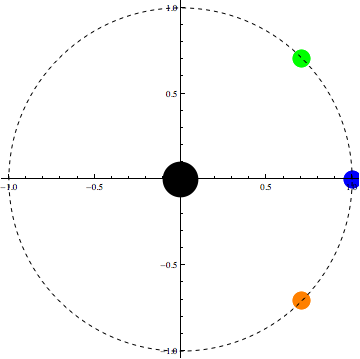}
\includegraphics[width=1in]{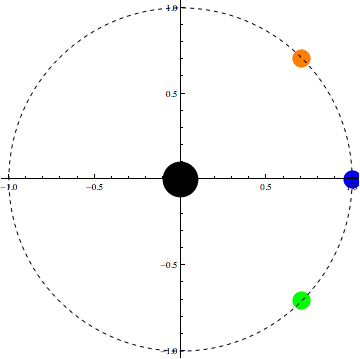}
\includegraphics[width=1in]{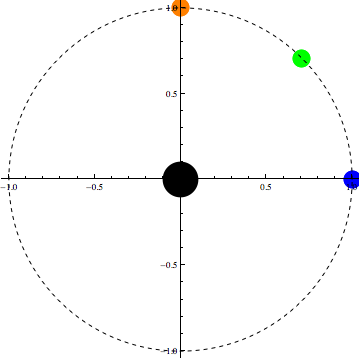}
\includegraphics[width=1in]{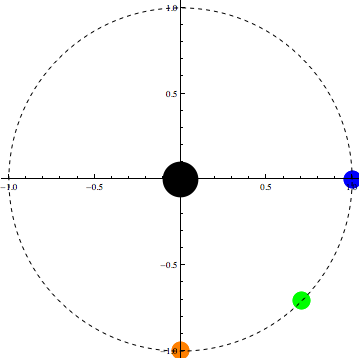}
\caption{Family 3: Six minima of $V$.  These configurations are stable.}
\label{family3}
\end{subfigure}
\caption{All distinct \added{(up to rotational symmetry)} critical points $(\theta_1,\theta_2,\theta_3)$ of $V$ for the case of equal weights.   \added{Since the weights are equal, any reordering of a critical point produces another critical point. We identify three families up to ordering.}\added{ Theorem \ref{thm:existence} implies that these critical points are limiting configurations for sequences of relative equilibria of the $4$-vortex problem as $\varepsilon_k\rightarrow 0$, and Theorem \ref{stability theorem} determines the linear stability of members of the sequence.  We therefore refer to the limit of a stable (unstable) sequence as stable (unstable) without ambiguity.  In all figures, the blue dot corresponds to $\theta_1$, the orange dot to $\theta_2$, and the green dot to $\theta_3$.}}
\label{example1families}
\end{figure}
\FloatBarrier
\subsection{Asymmetric Equilibria\label{sect:example 2}}

\added{Theorem \ref{symmetry theorem} implies that all critical points of $V$ are asymmetric when $\mu_i\neq \mu_j,\, i\neq j$.\added{ In this section, we show that there are \textit{stable} asymmetric relative equilibria and present a number of examples.} We first consider a case where $\mu_i>0$ for all $i$, hence $V$ once again has the limit potential property. Then we consider two cases where circulation weights have varying signs. Additionally, we use an asymmetric critical point of $V$ as a starting point for finding stable asymmetric relative equilibria of the full $4$-vortex problem when $\varepsilon>0$, see Figure \ref{example3continuations}.  }\added{As before, $\theta_1$ is fixed at 0 to reduce rotational symmetry of $V$. The position of $\theta_1$ is colored blue, $\theta_2$ is orange, and $\theta_3$ is green.}

For the first asymmetric example, let $(\mu_1,\mu_2,\mu_3)=(2,1,9)$. The Hermite method yields 10 real valued critical points of $V$, \added{and all were found to be nondegenerate}. \added{In Figure \ref{example2families}, we have grouped the 10 configurations into 5 distinct families of relative equilibria. Each family contains two configurations which are the same up to ordering in $\theta_i$.}   \added{Since all $\mu_i$ are positive, $V$ has the property that minima are limits of sequences of stable relative equilibria of the $4$-vortex problem.  Since these critical points are asymmetric, we immediately have the following result:}

\begin{thm}
There exist linearly stable relative equilibria of the $4$-vortex problem without a line of symmetry.
\end{thm}

\begin{figure}[ht]
\centering
\begin{subfigure}[t]{.45\textwidth}
\centering
	\includegraphics[width=1in]{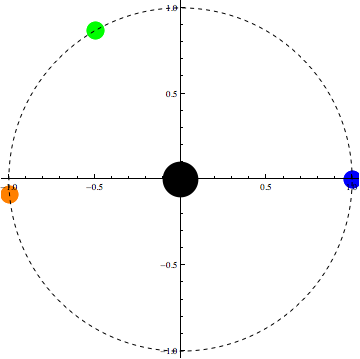} 
	\includegraphics[width=1in]{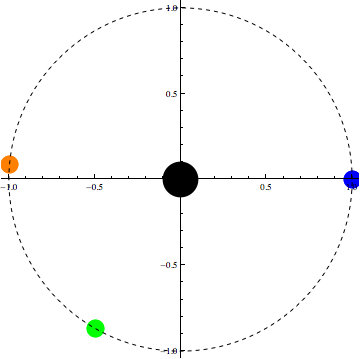}
	\caption{Family 1: Saddle points\\Unstable}
	\end{subfigure}
	\begin{subfigure}[t]{.45\textwidth}
	\centering
		\includegraphics[width=1in]{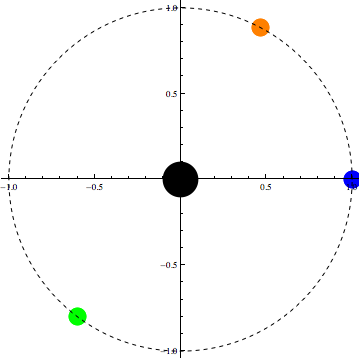}
	 \includegraphics[width=1in]{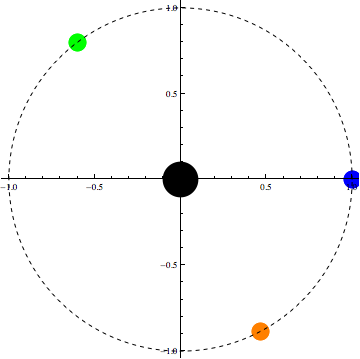}
	 \caption{Family 2: Saddle points\\ Unstable}
	 \end{subfigure}
\begin{subfigure}[t]{.45\textwidth}
\centering
\includegraphics[width=1in]{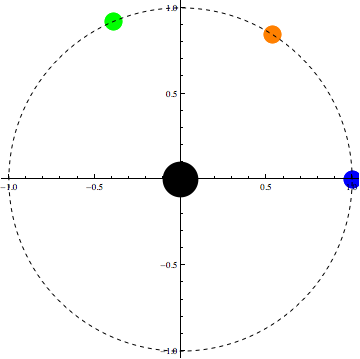}
\includegraphics[width=1in]{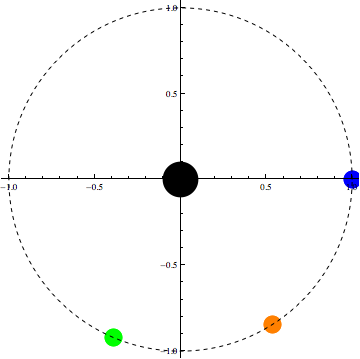}
\caption{Family 3: Minima\\ Stable}
\end{subfigure}
\begin{subfigure}[t]{.45\textwidth}
\centering
\includegraphics[width=1in]{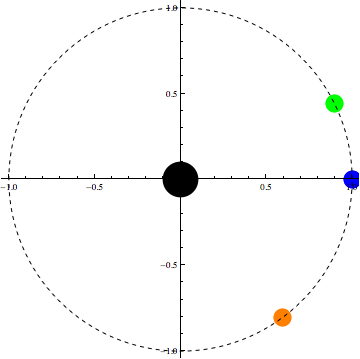} 
\includegraphics[width=1in]{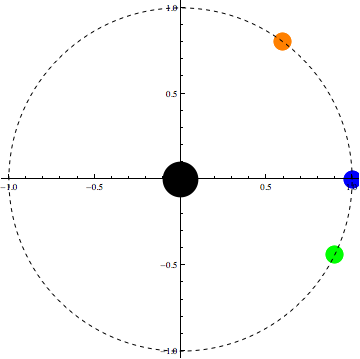}
\caption{Family 4: Minima\\Stable}
\end{subfigure}
\begin{subfigure}[t]{.45\textwidth}
\centering
\includegraphics[width=1in]{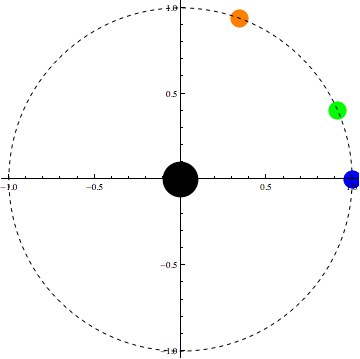}
\includegraphics[width=1in]{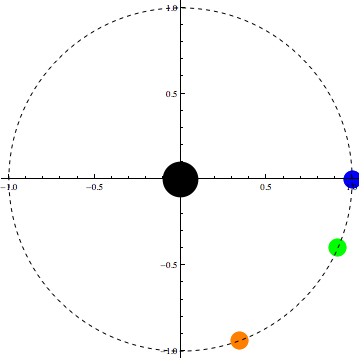}
\caption{Family 5: Minima\\Stable}
\end{subfigure}
\caption{Five families of critical points of $V$ that are relative equilibria of the $(1+3)$-vortex problem when $(\mu_1,\mu_2,\mu_3)=(2,1,9)$, i.e. they are limits of relative equilibria of the $4$-vortex problem as $\varepsilon\rightarrow 0$.}
	\label{example2families}
\end{figure}

\FloatBarrier
 We now consider an example with both positive and negative circulation weights.  Let $(\mu_1,\mu_2,\mu_3)=(2,-1,3)$. As in \added{the previous example}\deleted{in Example 2} there are no symmetric configurations. There are again 10 real-valued critical points of $V$, and 5 distinct families of relative equilibria up to ordering. These are pictured in Figure \ref{example3families}. Because the signs of the circulation parameters are mixed, $\mu^{-1}$ is no longer a positive-definite metric on the configuration space, and it is possible to have saddle points that are limits of sequences of stable relative equilibria.

\begin{figure}[ht]
\centering
\begin{subfigure}[t]{.45\textwidth}
\centering
		\includegraphics[width=1in]{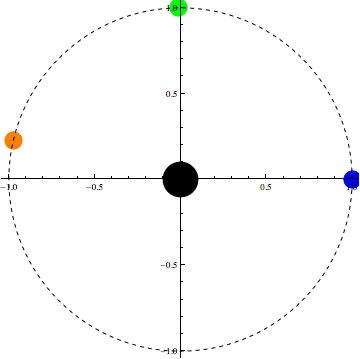} 
		\includegraphics[width=1in]{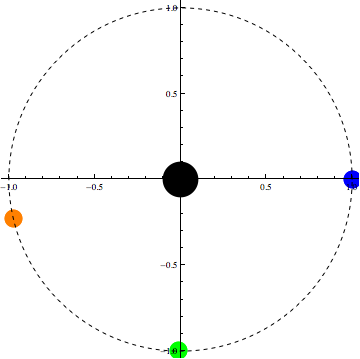}
		\caption{Family 1: Maxima\\ Unstable}
	\end{subfigure}
	\begin{subfigure}[t]{.45\textwidth}
	\centering
		\includegraphics[width=1in]{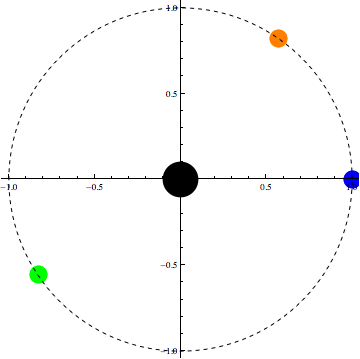}
	 \includegraphics[width=1in]{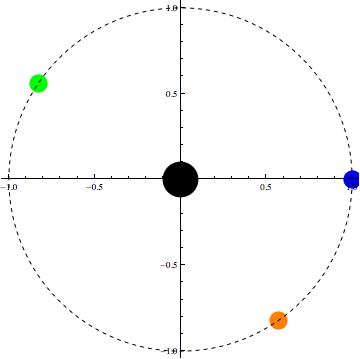}
	 \caption{Family 2: Minima\\Unstable}
	 \end{subfigure}
	 \begin{subfigure}[t]{.45\textwidth}
	 \centering
		\includegraphics[width=1in]{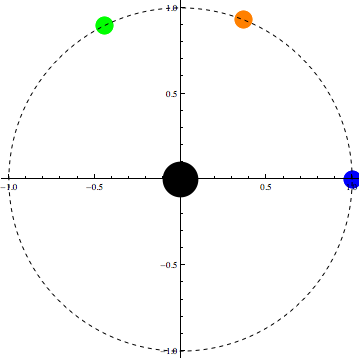} 
		\includegraphics[width=1in]{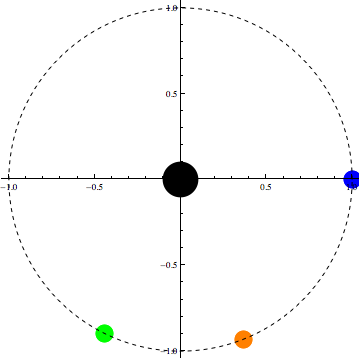}
		\caption{Family 3: Saddle points \\Unstable}
		\end{subfigure}
		\begin{subfigure}[t]{.45\textwidth}
		\centering
			\includegraphics[width=1in]{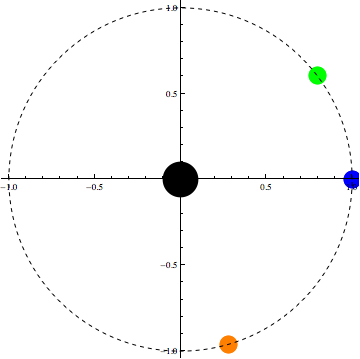} 
		\includegraphics[width=1in]{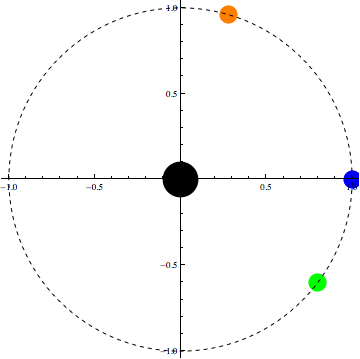}
		\caption{Family 4: Saddle points\\Unstable}
		\end{subfigure}
		\begin{subfigure}[t]{.45\textwidth}
		\centering
		\includegraphics[width=1in]{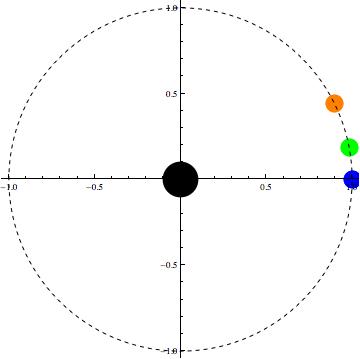}
		\includegraphics[width=1in]{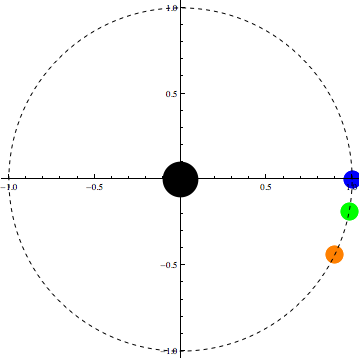}
		\caption{Family 5: Saddle points\\Stable}
		\end{subfigure}
		\caption{Five families of critical points of $V$ that are relative equilibria of the $(1+3)$-vortex problem when $(\mu_1,\mu_2,\mu_3)=(2,-1,3)$ }
		\label{example3families}
\end{figure}

\added{Because nondegenerate critical points of $V$ are limits of sequences of relative equilibria of the point vortex equations, they can be used as starting points for numerical continuation.  We illustrate this idea in  Figure \ref{example3continuations}.  The first panel shows a nondegenerate critical point of $V$ satisfying the hypotheses of Theorem \ref{stability theorem}, i.e. $\mu^{-1}V_{\theta\theta}$ has $N-1$ positive eigenvalues. The remaining panels show members of a ``nearby'' family of asymmetric linearly stable relative equilibria. To find these relative equilibria, we use a numerical root finder based on Newton's method to identify zeroes of the full point vortex equations with $\varepsilon$ small and the critical point as the initial guess. We then increase $\varepsilon$ incrementally and apply the procedure repeatedly using the relative equilibrium from the previous step as the new initial guess.} \deleted{, we work with the normalized set $(\mu_1,\mu_2,\mu_3)=\frac{1}{\sqrt{14}}(2,-1,3)$. We ``grow" a family of relative equilibria of the 4-vortex problem using a critical point of $V$ as the starting point.} 
\\
\begin{figure}[ht]
\centering
\begin{subfigure}[t]{.3\textwidth}
\centering
\includegraphics[width=1in]{Comp3bRE9.png}
\caption{Critical point of $V$}
\end{subfigure}
\begin{subfigure}[t]{.3\textwidth}
\centering
\includegraphics[width=1in]{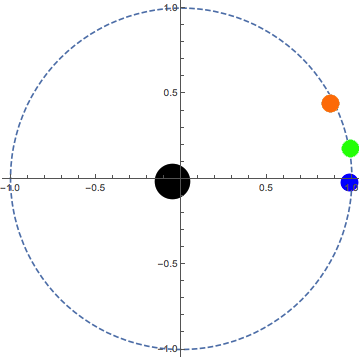}
\caption{$\varepsilon=.05$}
\end{subfigure}
\begin{subfigure}[t]{.3\textwidth}
\centering
\includegraphics[width=1in]{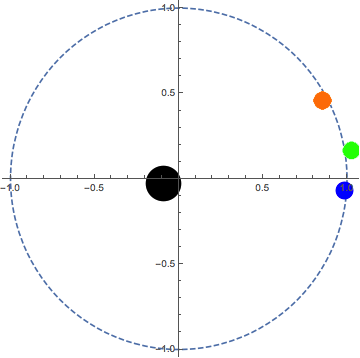}
\caption{$\varepsilon =.1$}
\end{subfigure}
\caption{\added{Panel (a) shows a nondegenerate asymmetric saddle point of $V$ with $(\mu_1,\mu_2,\mu_3)=\frac{1}{\sqrt{14}}(2,-1,3)$ such that $\mu^{-1}V_{\theta\theta}$ has two positive eigenvalues and one zero eigenvalue, hence Theorem 3 implies the existence of a convergent sequence of stable relative equilibria of the $4$-vortex problem.  Nearby stable asymmetric relative equilibria of the $4$-vortex problem for $\varepsilon>0$ are shown in (b) and (c).} }
\label{example3continuations}
\end{figure}
\FloatBarrier 
Let $(\mu_1,\mu_2,\mu_3)=(-1,-3,10)$. There are 8 real critical points of $V$ and 4 distinct families up to ordering of the vortices that are pictured in Figure \ref{example4families}. This example has maxima of $V$ that correspond to stable relative equilibria. We remark that if all $\mu_i$ were negative, $V$ would be a ``negative potential,'' and maxima of $V$ would be stable.
\begin{figure}[ht]
	\centering
	\begin{subfigure}[t]{.45\textwidth}
	\centering
		\includegraphics[width=1in]{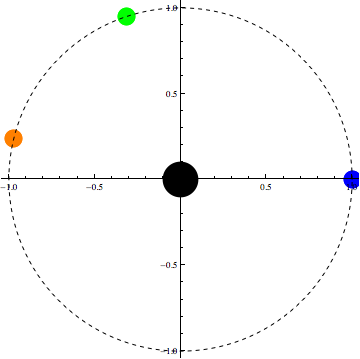}
		 \includegraphics[width=1in]{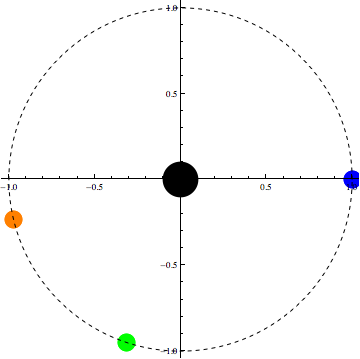}
		 \caption{Family 1: saddle\\ not stable}
		 \end{subfigure}
	\begin{subfigure}[t]{.45\textwidth}
	\centering
			\includegraphics[width=1in]{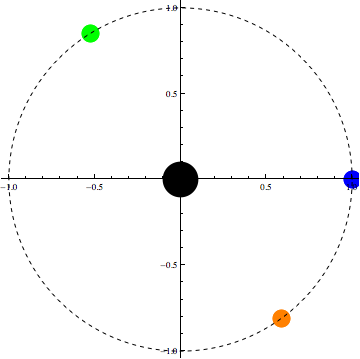}
		 \includegraphics[width=1in]{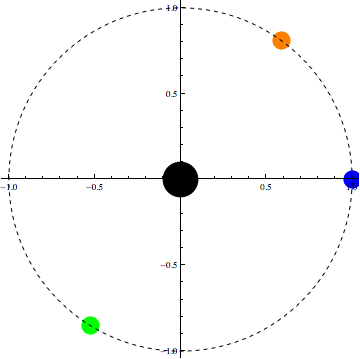}
		 \caption{Family 2: saddle\\not stable}
		 \end{subfigure}
	\begin{subfigure}[t]{.45\textwidth}
	\centering
		\includegraphics[width=1in]{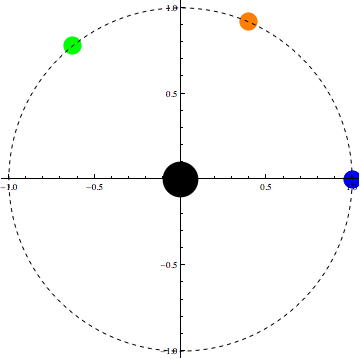}
		\includegraphics[width=1in]{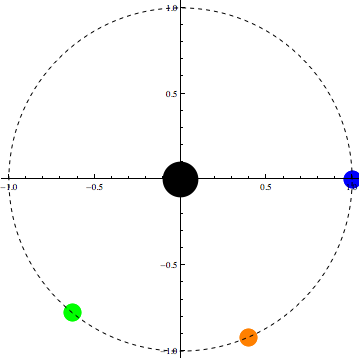}
		\caption{Family 3: maximum\\ stable}
		\end{subfigure}
	\begin{subfigure}[t]{.45\textwidth}
	\centering
		\includegraphics[width=1in]{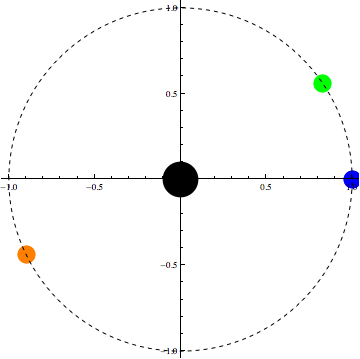} 
		\includegraphics[width=1in]{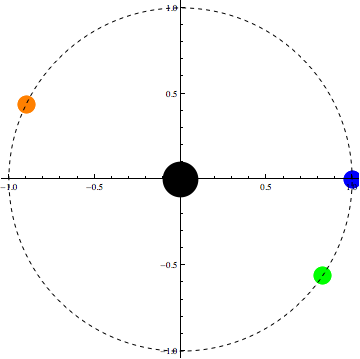}
		\caption{Family 4: minimum\\ not stable}
		\end{subfigure}
\caption{Four families of critical points of $V$ that are relative equilibria of the $(1+3)$-vortex problem when $(\mu_1,\mu_2,\mu_3)=(-1,-3,10)$}
\label{example4families}
\end{figure}

\FloatBarrier
\section{Discussion\label{sect6}}
We have analyzed existence and stability of relative equilibria with a dominant vortex in the case where weak vortices are allowed to have circulations of varying signs and weights. First, we extended the work of \cite{barry2012relative} to construct a particular function $V$ whose critical points are exactly the angular positions of relative equilibria in the $(1+N)$-vortex problem, the limiting relative equilibrium problem for one dominant vortex and $N$ infinitesimal vortices. Under nondegeneracy conditions, critical points can be continued to relative equilibrium solutions of the full point vortex equations. Additionally, for sufficiently small $\varepsilon$, the linear stability of these families of relative equilibria is determined by the eigenvalues of the circulation-weighted Hessian matrix $\mu^{-1}V_{\theta\theta}$. When the circulation weights are positive, minima of $V$ are limits of linearly stable configurations, hence $V$ is described as a ``limit potential". The potential property does not hold when some circulations are negative. 

In the 4-vortex problem with one dominant vortex, we discovered that symmetry of relative equilibria is actually rare. In the weak vortex circulation limit,
a configuration can be symmetric only when there is a symmetry in the vortex strengths -- at least two must be equal.  We were able to successfully implement analytical techniques from algebraic geometry for identifying and counting roots, and this approach was partly motivated by the expectation that they will be generalizable and useful for the analysis of similar problems with larger $N$. \added{In several cases} we identified all relative equilibria and classified their stability.  We were further able to illustrate the existence of linearly stable asymmetric relative equilibria.  \added{Related methods have proved fruitful for the study of bifurcations in the $(1+3)$-body problem \cite{corbera2011central}. Work in preparation will show how the methods used here, particularly the Hermite method, can be used to study bifurcations of relative equilibria in this problem, see \cite{hoyer2014bifurcations}.} Moreover, we are currently exploring the application of these methods to Bose-Einstein Condensate point vortex models, see e.g. \cite{barry2015generating} in which equilibria were described as roots of generating polynomials.

\textbf{Acknowledgements:} The authors thank Chris Budd, Rachel Kuske, and Rick Moeckel for helpful discussions and suggestions, and are especially grateful to Rick for his Mathematica code for the Hermite Method.  A.M.B. acknowledges support from the National Science Foundation grant DMS-1402372 and the Institute for Mathematics and its Applications.  A.H-L. was supported by Rick Moeckel's NSF grant DMS-1208908 and as an Ed Lorenz Postdoctoral Fellow by the Mathematics and Climate Research Network with funds provided by NSF DMS-0940243.

\bibliographystyle{siam}
\bibliography{RE2}

\appendix
\section{Mathematica Code for Hermite Algorithm\label{Appendix}}
Thanks to Rick Moeckel for providing us with this code.
\begin{verbatim}
GetExponents[mon_, vars_] := Map[Exponent[mon, #] &, vars]
GetAllExponents[poly_, vars_] := With[{poly1 = Expand[poly]},
  Table[GetExponents[poly[[i]], vars], {i, 1, Length[poly]}] // Union]

LeadingDegRevLexExponent[explist_] := Module[{maxdegree, l},
  deg[v_] := v[[1]] + v[[2]];
  maxdegree = Max[Map[deg, explist]];
  l = Select[explist, (deg[#] == maxdegree) &];
  Sort[l, (#1[[2]] < #2[[2]]) &][[1]]]

MakeCone[{a_, b_}, cmax_, dmax_] := 
 Flatten[Table[{c, d}, {c, a, cmax}, {d, b, dmax}], 1]

MonomialBasis[lexps_, cmax_, dmax_] := Module[{l, i},
  l = MakeCone[{0, 0}, cmax, dmax];
  For[i = 1, i <= Length[lexps], i++,
   l = Complement[l, MakeCone[lexps[[i]], cmax, dmax]];];
  l]
  
GetCoefficient[p_, {a_, b_}] := Module[{q},
  q = Expand[p];
  If[{a, b} == {0, 0}, q /. {r1 -> 0, r2 -> 0},
   If[a == 0, Coefficient[q, r2^b] /. {r1 -> 0},
    If[b == 0, Coefficient[q, r1^a] /. {r2 -> 0},
     Coefficient[q, r1^a*r2^b]]]]
  ]


PolyReduceVector[p_, gb_, mb_] := Module[{r},
  r =  PolynomialReduce[p, gb, {r1, r2}, 
     MonomialOrder -> DegreeReverseLexicographic][[2]];
  Map[GetCoefficient[r, #] &, mb]]

PolyReduceCoefficient[p_, gb_, {a_, b_}] := Module[{r},
  r =  PolynomialReduce[p, gb, {r1, r2}, 
     MonomialOrder -> DegreeReverseLexicographic][[2]];
  GetCoefficient[r, {a, b}]]

MultMapTrace[f_, gb_, mb_] := 
 Plus @@ Table[
   PolyReduceCoefficient[f*r1^mb[[i, 1]]*r2^mb[[i, 2]], gb, 
    mb[[i]]], {i, 1, Length[mb]}]

HermiteForm[q_, gb_, mb_] := Module[{H, i, j, mon},
  H = Table[0, {i, 1, Length[mb]}, {j, 1, Length[mb]}];
  For[i = 1, i <= Length[mb], i++,
   H[[i, i]] = 
    MultMapTrace[q*r1^(2 mb[[i, 1]])*r2^(2*mb[[i, 2]]), gb, mb];
   (*Print[{i,i},H[[i,i]]];*)
   For[j = i + 1, j <= Length[mb], j++,
    mon = 
     q*r1^(mb[[i, 1]] + mb[[j, 1]])*r2^(mb[[i, 2]] + mb[[j, 2]]);
    H[[i, j]] = MultMapTrace[mon, gb, mb];
    H[[j, i]] = H[[i, j]];
    (*Print[{i,j},{j,i},H[[i,j]]];*)];];
  H]

ClearCol[A_, i_] := Module[{B = A, d, ri, j},
  ri =  A[[i]];
  d = ri[[i]];
  For[j = i + 1, j <= Length[A], j++,
   B[[j]] = B[[j]] - B[[j, i]]*ri/d;];
  B]

ClearRow[A_, i_] := Transpose[ClearCol[Transpose[A], i]]
ClearCR[A_, i_] := ClearRow[ClearCol[A, i], i]

SwapRow[A_, i_, j_] := Module[{B = A},
  B[[i]] = A[[j]];
  B[[j]] = A[[i]];
  B]

SwapCol[A_, i_, j_] := Transpose[SwapRow[Transpose[A], i, j]]
SwapCR[A_, i_, j_] := SwapRow[SwapCol[A, i, j], i, j]

RowSumDiff[A_, i_, j_] := Module[{B = A, v, w},
  v = A[[i]] + A[[j]];
  w = -A[[i]] + A[[j]];
  B[[i]] = v;
  B[[j]] = w;
  B]
Clear[ColSumDiff]
ColSumDiff[A_, i_, j_] := Transpose[RowSumDiff[Transpose[A], i, j]]
SumDiffCR[A_, i_, j_] := RowSumDiff[ColSumDiff[A, i, j], i, j]

SymmetricReduce[A_] := Module[{B = A, n = Length[A], i, j, k},
  For[i = 1, i <= n, i++,
   (* if pivot is 0 look for a nonzero diagonal and switch*)
   
   If[B[[i, i]] == 0, For[j = i + 1, j <= n, j++,
      If[B[[j, j]] != 0, B = SwapCR[B, i, j]; Break[];]; ];];
   (* if pivot is still 0 do a row/col sum/diff *)
   
   If[B[[i, i]] == 0, For[j = i + 1, j <= n, j++,
      If[B[[i, j]] != 0, B = SumDiffCR[B, i, j]; Break[];];];];
   If[B[[i, i]] != 0, B = ClearCR[B, i];];
   ];
  B]

Sig[A_] := Module[{A1, i, diagsigns, p, n},
  A1 = SymmetricReduce[A];
  diagsigns = Table[Sign[A1[[i, i]]], {i, 1, Length[A1]}];
  Plus @@ diagsigns]
  
\end{verbatim}
\end{document}